\documentclass[article]{amsart}
\usepackage{amsbsy}
\usepackage{amsmath}
\usepackage{amstext}
\usepackage{amsthm}
\usepackage{amscd}
\usepackage{amssymb}
\usepackage{amsfonts}

\theoremstyle{plain}
\newtheorem{thm}{Theorem}[section]
\newtheorem{cor}[thm]{Corollary}

\newtheorem{prop}[thm]{Proposition}

\theoremstyle{definition}

\theoremstyle{remark}

\numberwithin{equation}{section}

\newcommand{\hide}[1]{{}}

\providecommand{\keywords}[1]{\textbf{\textit{Index terms---}} #1}

\usepackage{etoolbox}
\AtBeginEnvironment{array}{\setlength{\arraycolsep}{50pt}}
\begin{document}

\title[Non-vanishing of Modular $L$-values and Fourier coefficients]{On the non-vanishing of modular $L$-values and Fourier coefficients of cusp forms}
\author{Jun hwi Min}
\email{beliefonme159@unist.ac.kr}
\address{Ulsan National Institute of Science and Technology, Ulsan, Korea}

\begin{abstract}
	We prove a non-vanishing result of modular $L$-values with quadratic twists, where the quadratic discriminants are in a  short interval.  Using this fact and Waldspurger's theorem, we improve the results of Balog-Ono[The chebotarev density theorem in short intervals and some questions of Serre, Journal of number theory. 91(2):356-371(2001)] on the non-vanishing of Fourier coefficients of half-integral weight eigenform.  
\end{abstract}
\maketitle

	\section{Introduction and statement of results}

		  Let $f$ be a Hecke eigenform of level $N$ and weight $2k$ where $k$ is a positive integer, and let $\chi_d$ be a primitive quadratic character associated to the quadratic field $\mathbb{Q}(\sqrt{d})$.  The study of the non-vanishing of $L(k,f,\chi_d)$ arose from the study of the rank of modular elliptic curves. The celebrated theorem of Kolyvagin asserted that if $L(s,E)$ has a simple zero at $s=1$ and there exists a quadratic discriminant $d<0$ coprime to $4N$ such that $L(1,E,\chi_d)\neq 0$, then the rank$E(\mathbb{Q})=1$(See the introduction of \cite{pomykala1997non}). Murty-Murty\cite{murty1991mean} showed that one can get rid of the non-vanishing condition, by showing that there exists an infinite family of quadratic discriminants with $L(1,E,\chi_d)\neq 0$. Even further, the quantitative estimates on the number of quadratic discriminants with $L(1,E,\chi_d)\neq 0$ has gained huge interests. Let 
		  \begin{align*}
		  	N_f(X):=\#\{d:d \text{ is a square-free quadratic discriminant},  |d|\leq X, L(k,f,\chi_d)\neq 0 \}.
		  \end{align*}
	 
	 The well-known conjecture on the non-vanishing of $L(k,f,\chi_d)$ asserts that $N_f(X)\gg X$. The best known lower bound is that of Ono-Skinner\cite{ono1998non}; when $f$ is a even weight newform of trivial nebentype, 
	 	 \begin{align*}
	 	N_f(X)\gg X/\log X.
	 \end{align*}	 
 
  	We are interested in a short interval analogue of $N_f(X)$, where $f(z)\in S_{2k}(N,\psi)$ is a normalized Hecke eigenform. Here, the set of quadratic discriminants are defined modulo $4N$, Namely
  	\begin{align*}
  		\mathcal{D}:=\{0<(-1)^kd: d\equiv v^2 \mod 4N \text{ for some } (v, 4N)=1\}.
  	\end{align*}
  	We denote $\gamma(4N):=\#\mathcal{D}$.
  	The counting function of the number of nonzero $L$-values in short interval is defined by
  	\begin{align*}
  		N_f(X,h):=\#\{d\in \mathcal{D}:d \text{ square-free}, X\leq |d|\leq X+h, L(k,f,\chi_d)\neq 0 \}.
  	\end{align*}	 
  	
    To bound $N_f(X,h)$, it is necessary to estimate the first moment of $L$-values in a short interval. That is,     
    \begin{align*}
    	S_f(X,h):&=\sideset{}{'}\sum_{\substack{d\in \mathcal{D}\\\ X\leq |d| \leq X+h}}L(k,f,\chi_d).
    \end{align*}
  		Here, the  superscript $'$ denotes that the sum is taken over the square-free numbers $d$.
		
		\begin{thm}\label{firstmoment}
			 For $X^{3/4+\epsilon}\leq h \leq X$, we have
			\begin{align*}
				S_f(X,h)=C_{N}L_f(k)h+O_{f,\epsilon}(hX^{-\epsilon}),
			\end{align*}
		 where the constants $C_N$ and $L_f(k)$ are given by 
			\begin{align*}
				C_N=\frac{3\gamma(4N)}{\pi^2N}\prod_{p|4N}(1-p^{-2})^{-1}
			\end{align*}
			and 
			\begin{align}
				L_f(k)=\sum_{\substack{n=rj^2}}\frac{a_n}{n^k}\prod_{\substack{p|n\\(p,4N)=1}}(1+p^{-1})^{-1}\neq 0,\label{lvalue}
			\end{align}
			respectively. In the R.H.S of (\ref{lvalue}), the indices $r$ and $j$ are positive integers with $r|(4N)^\infty$ and $(j,4N)=1$, respectively.  Note $L_f(s)$ is the $L$-function $L_{f,1}(s)$ given in \cite[p.385]{luo1997determination}. 
		\end{thm}
		
		To obtain Theorem \ref{firstmoment}, we need to make use of the second moment of $L$-values, given in \cite{perelli1997averages} and \cite{justus2008moments}. 
		\begin{thm}\cite{perelli1997averages,justus2008moments}\label{secondmoment}
			If $f$ is an eigenform, then
				\begin{align*}
				\sideset{}{'}\sum_{\substack{|d| \leq X}}|L(k,f,\chi_d)|^2\ll_\epsilon X^{1+\epsilon}.
			\end{align*}
		\end{thm}
	 	The authors in \cite{perelli1997averages} and \cite{justus2008moments} proved Theorem \ref{secondmoment} for newforms, but the discussion goes through for eigenforms as well.  
		
		By Cauchy's inequality, we have an immediate corollary. 
		\begin{cor}\label{nonvanishing} For $X^{3/4+\epsilon}\leq h\leq X$, we have
			\begin{equation*}
				N_f(X,h)\gg \frac{h^2}{X^{1+\epsilon}}.
			\end{equation*}
			The implied constant only depends on $f,\epsilon$.
		\end{cor}

		As far as we know, the only present result on non-vanishing of $L(k,f,\chi_d)$ in short intervals other than Corollary \ref{nonvanishing} is that of Balog-Ono\cite{balog2001chebotarev}. Their result is as follows. 
		\begin{thm}\cite{balog2001chebotarev}\label{ono1}
			 Let $f\in S_{2k}^{new}(\Gamma_0(N),\psi_{triv})$. If $f(z)$ is not a linear combination of weight-$3/2$ theta functions, then there exists a positive integer $k_f$ such that for $ X^{1-1/k_f+\epsilon}\leq h \leq X$,
			\begin{equation}
				\#\{d: d \text{ square-free}, X\leq |d|\leq X+h,L(k,f,\chi_d)\neq 0\}\gg \frac{h}{\log X}.\label{onoequation}
			\end{equation}
		\end{thm}
		
		Note that in (\ref{onoequation}), $d$ runs through not only $\mathcal{D}$ but also all square-free quadratic discriminants. Unlike Theorem \ref{ono1}, Corollary \ref{nonvanishing} gives a non-vanishing result for the eigenforms of non-trivial nebentypus. In case of trivial nebentype, although Corollary \ref{nonvanishing} gives a weaker result when $h\geq X^{1-1/k_f+\epsilon}$, it is worth noting that it gives further information on $N_f(X,h)$ when $X^{3/4+\epsilon}\leq h< X^{1-1/k_f+\epsilon}$. The advantage of having a non-vanishing result in a shorter inteval is more visible  in view of the non-vanishing of Fourier coefficients. 
		
		 Theorem \ref{ono1} is a consequence of the following corollary and Waldspurger's theorem.
		
		\begin{thm}\cite[Corollary 4]{balog2001chebotarev}\label{ono2}
			Let $g\in S_{k+1/2}(M, \psi)$, i.e., a cuspform of level $M$ and weight $k+1/2$ with the nebentype $\psi$. If $g(z)$ is not a linear combination of weight-$3/2$ theta functions, there exists  a positive integer $k_f$ such that for $ X^{1-1/k_f+\epsilon}\leq h \leq X$,
			\begin{equation*}
				\#\{ X\leq n\leq X+h:a_g(n)\neq 0 \}\gg \frac{h}{\log X}.
			\end{equation*}
		\end{thm}

		\begin{thm}\cite{pacetti2008shimura}\label{walds}
			Let $N$ be even, $f\in S_{2k}(N,\psi^2)$ and $g\in S_{k+1/2}(2N,\psi)$ where $f$ is the Shimura correspondent of $g$. Then 
			\begin{equation*}
				a_g(d)^2=\kappa_f L(k,f,\psi_0^{-1}\chi_d)\psi(d)d^{k},
			\end{equation*}
			where $\psi_0(n)=\psi(n)(\frac{-1}{n})^{k}$.
		\end{thm}

		Even further, Theorem \ref{ono2} answers to a stronger form of Serre's questions. Serre initiated the study on the sizes of gaps between non-zero Fourier coefficients. He defined the following gap function.
		\begin{align*}
			i_f(n)=\begin{cases}
				\max\{i:a_f(n+j)=0\text{ for all }0\leq j\leq i\} &\text{if }  a_f(n)=0,\\
				0&\text{otherwise}.
			\end{cases}
		\end{align*}
	
		Serre's question asks us to estimate this gap function with regards to the cusp forms of integral or half-integral weights. Combining Theorem \ref{walds} and Corollary \ref{nonvanishing}, our second corollary immediately follows.
		\begin{cor}\label{halfcorollary}
			Assume the conditions in Theorem \ref{walds} with $\psi=\psi_{triv}$. Then for $ X^{3/4+\epsilon}\leq h\leq X$,
			\begin{align*}
				\#\{X\leq n\leq X+h: a_g(n)\neq 0 \}\gg_{f,\epsilon} \frac{h^2}{X^{1+\epsilon}}.
			\end{align*}
			In particular, $i_g(n)\ll n^{3/4+\epsilon}$.
		\end{cor}
		
		Corollary \ref{halfcorollary} along with Theorem \ref{ono2} are the only results on $i_g(n)$ where $g$ is of half-integral weight.  On the other hand, many mathematicians have studied $i_f(n)$ where $f$ is of integral weight. Unfortunately, most of the approaches in the study of integral weight case are not available in half-integral weight case. Let us briefly review the works on $i_f(n)$.  
		
		Rankin-Selberg estimates, the multiplicativity of Hecke operators, or the arithmetic of Galois representations have been useful tools to study $i_f(n)$ for the eigenforms of integral weight. For example, we have a classical result on $L(f\otimes \overline{f})$ such that there exists an integer $c_f$ for which
		\begin{align*}
			\sum_{n\leq X}|a_f(n)|^2n^{1-\frac{k}{2}}= c_f X+O(X^{\frac{3}{5}}).
		\end{align*}
		It immediately follows that $i_f(n)\ll n^{3/5}$. Another approach appeals to the theory of $\mathfrak{B}$-free numbers, especially when $f$ is not of CM type.  Let $\mathfrak{B}=\{b_i\}$ be a set of integers such that 
		\begin{align*}
			\sum_{b\in \mathfrak{B}}\frac{1}{b}<\infty\text{ and }(b_i,b_j)=1 \text{ whenever }i\neq j.
		\end{align*}
		We say that a natural number $n$ is $\mathfrak{B}$-free if it is not divisible by any of the elements of $\mathfrak{B}$. Specifically, one can define $\mathfrak{B}$ as follows. 
		\begin{align*}
			\mathfrak{B}=\{p \text{ prime }:a_f(p)=0\}\cup \{p |N \text{ prime}\}.
		\end{align*}
		$\mathfrak{B}$ has a zero density due to the following result of Serre\cite[p.174, Cor.2]{serre1981quelques}. Let $f$ be a newform with integral weight $2k\geq 2$ which is not of CM type. Then
		\begin{align}
			\#\{p\leq X \text{ prime }:a_f(p)=0\}\ll_{f,\epsilon}\frac{X}{(\log X)^{3/2-\epsilon}}.\label{serredensity}
		\end{align}
		
		In view of (\ref{serredensity}) and the multiplicativity of Hecke eigenvalues, $a_f(n)\neq 0$ if $n$ is square-free and $\mathfrak{B}$-free. Thus estimating $i_f(n)$ becomes a problem of counting $\mathfrak{B}$-free numbers in the short intervals. Balog-Ono were the first to take this approach and they deduced that 
		\begin{align*}
			i_f(n)\ll n^{17/41+\epsilon}.
		\end{align*} Later, the estimates of $i_f(n)$ has been refined several times. The best bound for $i_f(n)$ is due to Kowalski-Robert-Wu\cite{kowalski2007small}. They proved that for any holomorphic non-CM cuspidal eigenform $f$ on general
		congruence groups,
		\begin{align*}
			i_f(n)\ll n^{7/17+\epsilon}.
		\end{align*}
		
		If $f$ is of CM type, there are no similar general results on $i_f(n)$. The major difficulty in this case is that  the density estimate in (\ref{serredensity}) is valid only for non-CM type forms. Thus the previous works on CM type forms took different approaches other than distribution of $\mathfrak{B}$-free numbers. For example, Das-Ganguly\cite{das2014gaps} showed that for all nonzero cuspforms  $f$  of level one, 
		\begin{align}
			i_f(n)\ll n^{1/4}. \label{withcmlevelone}
		\end{align}
		The main ingredient of their work were the congruence relation of Hecke eigenvalues due to Hatada\cite{hatada1979eigenvalues} and the distribution of sum of two squares in short intervals. 
		
		(\ref{withcmlevelone}) can be extended to the eigenforms of higher levels, under some conditions on $f$. Let $E/\mathbb{Q}$ be an elliptic curve which has a cyclic rational $4$-isogeny and $f_E$ be a newform corresponding to $E$ by the modularity theorem. Kumar\cite{kumar2018gaps} first proved that $f_E$ satisfies (\ref{withcmlevelone}), by showing that there exists a positive integer $m$ such that $a_{f_E}(m)\neq 0$ and $m$ is a sum of two squares in intervals $(X,X+cX^{\frac{1}{4}})$.   From this, he deduced that if $f$ is $2$-adically close enough to $f_E$, $f$ also satisfies (\ref{withcmlevelone}).

		Let us recall some necessary facts for the future discussion. Let $e(z)=e^{2\pi i z}$ and
		\begin{align*}
			f(z)=\sum_{n\geq 1}a_ne(nz)
		\end{align*}
		be the Fourier expansion of $f$ at the cusp $\infty$.  The associated modular $L-$function $L(s,f)=\sum_{n\geq 1}a_nn^{-s}$  absolutely converges in $\Re(s)>k+1/2$. It can be analytically continued to an entire function and satisfies the following functional equation
		\begin{equation*}
			\bigg(\frac{\sqrt{N}}{2\pi}\bigg)^s\Gamma(s)L(s,f)
			=w\bigg(\frac{\sqrt{N}}{2\pi}\bigg)^{2k-s}\Gamma(2k-s)L(2k-s,f),
		\end{equation*}
		where  $w=\pm 1$. In addition, the twisted $L-$function $L(s,f,\chi_d)$ satisfies the functional equation
		\begin{equation*}
			\bigg(\frac{d\sqrt{N}}{2\pi}\bigg)^s\Gamma(s)L(s,f,\chi_d)
			=w_d\bigg(\frac{d\sqrt{N}}{2\pi}\bigg)^{2k-s}\Gamma(2k-s)L(2k-s,f,\chi_d),
		\end{equation*}
		where $w_d=w\chi_d(-N)=1$ for all $d\in \mathcal{D}$.
	 	
	 	If $f$ is an eigenform, we have the Euler product
	 	\begin{align}
	 		L(s,f)=\prod_p\bigg(1-\frac{\alpha_p}{p^s}\bigg)^{-1}\bigg(1-\frac{\beta_p}{p^s}\bigg)^{-1}\label{eulerproduct}
	 	\end{align}
 		and it satisfies the Ramanujan-Petersson conjecture (i.e. $|\alpha_p|=|\beta_p|=p^{k-1/2}$ for all $p\nmid N$ and $|\alpha_p|,|\beta_p|\leq p^{k-1/2}$ otherwise). Given $\alpha_p, \beta_p$ for all $p$, we define $n\mapsto \alpha_n,n\mapsto \beta_n$ as totally multiplicative functions on $\mathbb{N}$.
	\section{Proof of Theorem \ref{firstmoment}}

		We modify the methods of \cite{iwaniec1990order} and \cite{luo1997determination}. Necessary changes will be described in detail. 

	Set 
	\begin{align*}
		V(x)&:=\frac{1}{2\pi i}\int_{(4/5)}\frac{\Gamma(k+s)}{\Gamma(k)}x^{-s}\frac{ds}{s}.
	\end{align*}
	Here, note that the integral $\int_{(4/5)}$ denotes $\int_{4/5-i\infty}^{4/5+i\infty}$. By the Mellin transform and the integration by parts, we have
	\begin{align*}
		V(x)=\frac{1}{\Gamma(k)}\int_x^\infty e^{-y} y^{k-1} dy
		=(1+x+\dots+\frac{x^{k-1}}{(k-1)!})e^{-x}.
	\end{align*}
	
	Next, we will define $L(k,f,\chi_d)$ in terms of the rapidly convergent sums. Let
	\begin{align*}A(Q, \chi_d)&=\frac{1}{2\pi i}\int_{(4/5)}L(f, \chi_d, k+s)\frac{\Gamma(k+s)}{\Gamma(k)}\bigg(\frac{2\pi}{Q}\bigg)^{-s}\frac{ds}{s}.
	\end{align*}
	We have
	\begin{align*}A(Q, \chi_d)=\sum_{n\geq 1}a_n n^{-k} \chi_d(n) V\bigg(\frac{2\pi n}{Q}\bigg).
	\end{align*}
	By shifting the contour of integration to $\Re(s)=-4/5$, we obtain
	\begin{equation}
		L(k, f, \chi_d)=A(Q, \chi_d)+A(d^2NQ^{-1}, \chi_d),\label{L is sum of two terms}
	\end{equation} for any $Q>0$ and square-free $d\in D$. 
	In particular, 
	\begin{equation}\label{aval}
		L(k, f, \chi_d)=2A(|d|\sqrt{N}, \chi_d).
	\end{equation}

	By Abel's summation formula,
	\begin{equation}\label{abel}
		A(Q, \chi_d) \ll _fQ^{\frac{1}{2}}.
	\end{equation}
	Combining (\ref{L is sum of two terms}) and (\ref{abel}) gives 
	\begin{equation}
		L(k, f,\chi_d)=A(Q,\chi_d)+O_f(|d|Q^{-\frac{1}{2}}) \text{ for all } Q>0.\label{L is sum of A and O}
	\end{equation}

	As in \cite{luo1997determination}, we have an upper bound of the fourth moment of $L$-values 
	
	\begin{equation*}
		\sideset{}{'}\sum_{d\in D, |d|\leq X+h}|L(k,f,\chi_d)|^4  \ll (X+h)^{2+\epsilon}.
	\end{equation*}
	The first moment we are considering is
	\begin{equation*}
		S_f(X,h):=\sideset{}{'}\sum_{\substack{d\in D\\X\leq |d|\leq X+h}} L(k,f,\chi_d)
		=2\sideset{}{'}\sum_{\substack{d\in D\\X\leq |d|\leq X+h}} A(|d|\sqrt{N},\chi_d).
	\end{equation*}
	We introduce the M\"obius function to relax the square-free condition, so that 
	\begin{align*}
		S_f(X,h)&=2\sum_{\substack{d\in \mathcal{D}\\X/a^2\leq |d|\leq (X+h)/a^2}} \sum_{a^2|d}\mu(a)A(|d|\sqrt{N},\chi_d)\\
		&=2\sum_{(a,4N)=1}\sum_{\substack{d\in \mathcal{D}\\X/a^2\leq |d|\leq (X+h)/a^2}} \mu(a)A(a^2|d|\sqrt{N},\chi_{a^2d}).
	\end{align*}
	Here, $A$ is a large number only dependent on $X$. We choose $A$ later. We split the sum into two parts, say, $S_f(X,h)=S+R$, where
	
	\begin{align*}
		S&=2\sum_{\substack{a\leq A \\(a,4N)=1}}\mu(a)\sum_{\substack{d\in \mathcal{D}\\X/a^2\leq |d|\leq (X+h)/a^2}} A(a^2|d|\sqrt{N},\chi_{a^2d})					
	\end{align*}	
	and
	
	\begin{align*}
		R
		&=2\sum_{\substack{a> A \\(a,4N)=1}}\mu(a)\sum_{\substack{d\in \mathcal{D}\\X/a^2\leq |d|\leq (X+h)/a^2}} A(a^2|d|\sqrt{N},\chi_{a^2d})	\\
		&=2\sum_{(b,4N)=1}\bigg(\sum_{\substack{a|b\\a>A}}\mu(a)\bigg)\sideset{}{'}\sum_{\substack{d\in \mathcal{D}\\X/b^2\leq |d|\leq (X+h)/b^2}} A(b^2|d|\sqrt{N},\chi_{b^2d})	.				
	\end{align*}
	
	Now we first estimate the partial sum $R$. 
	
	\begin{prop}\label{R}
		$R  \ll _{f,\epsilon}A^{-1-\epsilon}(X+h)^{\epsilon+\frac{1}{2}}h^{\frac{1}{2}}+ (X+h)^{\frac{1}{2}}hA^{-3+\epsilon}$ 
	\end{prop}
	\begin{proof}
		From (\ref{eulerproduct}), (\ref{L is sum of A and O}) and by introducing the factors $\sum_{d_1|b}\mu(d_1),\sum_{d_2|b}\mu(d_2)$, as in \cite{iwaniec1990order} we have
	\begin{align*}
		A(Q,\chi_{b^2d})&=\sum_{\substack{d_1|b}}
		\sum_{\substack{d_2|b}} 
		\mu(d_1)\mu(d_2)
		\frac{\alpha_{d_1}\beta_{d_2}}{(d_1d_2)^{k}}
		\chi_d(d_1d_2)
		A(\frac{Q}{d_1d_2},\chi_d)\\
		&=\sum_{\substack{d_1|b}}
		\sum_{\substack{d_2|b}} 
		\mu(d_1)\mu(d_2)
		\frac{\alpha_{d_1}\beta_{d_2}}{(d_1d_2)^{k}}
		\chi_d(d_1d_2)
		\bigg(
		L(k,f,\chi_d)+O((d_1d_2)^{\frac{1}{2}}|d|Q^{\frac{-1}{2}})
		\bigg).
	\end{align*}
	We split the sum into $A,B$, where 
	\begin{align*}
		A=\sum_{\substack{d_1|b}}
		\sum_{\substack{d_2|b}} 
		\mu(d_1)\mu(d_2)
		\frac{\alpha_{d_1}\beta_{d_2}}{(d_1d_2)^{k}}
		\chi_d(d_1d_2)
		L(k,f,\chi_d)
	\end{align*} 
	and
	\begin{align*}
		B=\sum_{\substack{d_1|b}}
		\sum_{\substack{d_2|b}} 
		\mu(d_1)\mu(d_2)
		\frac{\alpha_{d_1}\beta_{d_2}}{(d_1d_2)^{k}}
		\chi_d(d_1d_2)
		O\bigg((d_1d_2)^{\frac{1}{2}}|d|Q^{-\frac{1}{2}}\bigg).
	\end{align*}
	 From Deligne's bound $\alpha_n,\beta_n\leq d(n)n^{(2k-1)/2}$, we have 
	\begin{align*}
		A\leq |L(k,f,\chi_d)|\sum_{\substack{d_1|b}}
		\sum_{\substack{d_2|b}} 
		(d_1d_2)^{-\frac{1}{2}+\epsilon} \ll |L(k,f,\chi_d)|
	\end{align*}
	and
	\begin{align*}
		B=O\bigg( |d|Q^{\frac{-1}{2}}\sum_{\substack{d_1|b}}
		\sum_{\substack{d_2|b}} 
		(d_1d_2)^{\frac{\epsilon}{2}}\bigg)
		=O\bigg(  b^\epsilon|d| Q^{\frac{-1}{2}}\bigg).
	\end{align*}
	Thus 
	\begin{equation*}
		A(Q,\chi_{b^2d})  \ll _f|L(k,f,\chi_d)|+ b^\epsilon|  d| Q^{-\frac{1}{2}}.
	\end{equation*}
	Collecting $A$ and $B$ together, and with $Q=b^2|d|\sqrt{N}$, $R$ has a bound
	\begin{align*}
		R&  \ll _f\sum_{(b,4N)=1}
		\bigg(\sum_{\substack{a|b\\a>A}}1\bigg)
		\sideset{}{'}\sum_{\substack{d\in \mathcal{D}\\X/b^2\leq |d|\leq (X+h)/b^2}} 
		\bigg(|L(k,f,\chi_d)|+ b^{-1+\epsilon}|  d|^{\frac{1}{2}}  \bigg).\\
	\end{align*}
	By the H\"older inequality and Theorem \ref{secondmoment}, 
	\begin{align*}
		&	\sideset{}{'}\sum_{\substack{d\in \mathcal{D}\\ X/b^2\leq |d|\leq (X+h)/b^2}} 
		|L(k,f,\chi_d)|\\
		&  \ll \bigg(\sideset{}{'}\sum_{\substack{d\in \mathcal{D}\\X/b^2\leq |d|\leq (X+h)/b^2}} |L(k,f,\chi_d)|^2\bigg)^\frac{1}{2}
		\bigg(\sum_{\substack{X/b^2\leq |d|\leq (X+h)/b^2}} 1\bigg)^\frac{1}{2}\\
		&  \ll _\epsilon((X+h)b^{-2})^\frac{1+2\epsilon}{2}h^{\frac{1}{2}}b^{-1}
		=b^{-2-2\epsilon}(X+h)^{\epsilon+\frac{1}{2}}h^{\frac{1}{2}}
	\end{align*}
	and 
	\begin{align*}
		&b^{-1+\epsilon}\sideset{}{'}\sum_{\substack{d\in \mathcal{D}\\ X/b^2\leq |d|\leq (X+h)/b^2}}|d|^\frac{1}{2}
		  \ll b^{-1+\epsilon}\frac{(X+h)^\frac{1}{2}}{b}\frac{h}{b^2}
		=(X+h)^\frac{1}{2}hb^{-4+\epsilon}.
	\end{align*}

	In sum,  
	\begin{align*}
		R&  \ll _f\sum_{(b,4N)=1}
		\bigg(\sum_{\substack{a|b\\a>A}}1\bigg)
		\bigg(b^{-2-2\epsilon}(X+h)^{\epsilon+\frac{1}{2}}h^{\frac{1}{2}}+ (X+h)^\frac{1}{2}hb^{-4+\epsilon}\bigg)\\
		&  \ll A^{-1-\epsilon}(X+h)^{\epsilon+\frac{1}{2}}h^{\frac{1}{2}}+ (X+h)^{\frac{1}{2}}hA^{-3+\epsilon}.
	\end{align*}	
\end{proof}

	We now evaluate $S$. For $(a, 4N)=1$ and $d\in \mathcal{D}$, we have
	\begin{equation*}
		A(a^2|d|\sqrt{N},\chi_{a^2d})=\sum_{(n,a)=1}a_nn^{-k}\chi_d(n)V(\frac{2\pi n}{a^2|d|\sqrt{N}}).
	\end{equation*}
	Write $n=rj^2m$, where $r|(4N)^\infty, (jm,4N)=1$, and $m$ is square-free.
	From
	\begin{equation*}
		\chi_d(n)=\chi_d(m) \text{ for } (d,j)=1,
	\end{equation*}
	we obtain
	\begin{align*}
		S&=2\sum_{\substack{a\leq A \\(a,4N)=1}}\mu(a)
		\sum_{\substack{n=rj^2m\\(n,a)=1}}a_nn^{-k}
		\sum_{\substack{(d,j)=1\\d\in \mathcal{D}\\X/a^2\leq |d|\leq (X+h)/a^2}} \chi_d(m)V(\frac{2\pi n}{a^2|d|\sqrt{N}})\\
		&=2\sum_{\substack{a\leq A \\(a,4N)=1}}\mu(a)
		\sum_{\substack{n=rj^2m\\(n,a)=1}}a_nn^{-k}
		\sum_{q|j}\mu(q)
		\sum_{\substack{dq\in \mathcal{D}\\X/a^2\leq |d|q\leq (X+h)/a^2}} \chi_{dq}(m)V(\frac{2\pi n}{a^2|d|q\sqrt{N}}).\\			
	\end{align*}
	In the second inequality, we introduced a M\"obius factor to relax the coprimality condition on $d$. 
	
	By the Gauss inversion formula,

	\begin{align*}
		S&=2\sum_{\substack{a\leq A \\(a,4N)=1}}\mu(a)
		\sum_{\substack{n=rj^2m\\(n,a)=1}}a_nn^{-k}
		\overline{\epsilon_{m}}{m}^{\frac{-1}{2}}
		\sum_{q|j}\mu(q)\\&\times
		\sum_{\substack{dq\in \mathcal{D}\\X/a^2\leq |d|q\leq (X+h)/a^2}} 
		\sum_{2|b|<{m}}\chi_{Nbq}({m})e\bigg(\frac{\overline{4N}bd}{m}\bigg)
		V\bigg(\frac{2\pi n}{a^2|d|q\sqrt{N}}\bigg),\\
	\end{align*}
	where
	\begin{align*}
		\overline{\epsilon_m}=
		\begin{cases}
			i &\text{ if } m\equiv -1\text{(mod 4)}\\
			1 &\text{ if } m\equiv 1\text{(mod 4)}.
		\end{cases}
	\end{align*}

	Set $\Delta=min(\frac{1}{2},a^2q(X+h)^{\epsilon-1})$. We split the sum $S$ into
	
	\begin{equation*}
		S=S_0+S_1+S_2,
	\end{equation*}
	where the three partial sums are restricted by the conditions $b=0, 0<|b|<\Delta m, \Delta m\leq |b|<m/2$, respectively.
	Evaluating $|S_1|$ and $|S_2|$ only requires minor modifications of  Iwaniec's method.
	
	First of all, note that as in \cite{iwaniec1990order},	$S_2 \ll _{f}1$.

	\begin{prop}\label{S_1}
			$S_1  \ll _fA^2(X+h)^{\epsilon-\frac{1}{2}}h$
	\end{prop}
\begin{proof}

		We have  
	\begin{align*}
		S_1=&2\sum_{\substack{a\leq A \\(a,4N)=1}}\mu(a)
		\sum_{\substack{rj^2\\(rj,a)=1}}
		\sum_{q|j}\mu(q)
		\sum_{\substack{dq\in \mathcal{D}\\X/a^2\leq |d|q\leq (X+h)/a^2}} \\&\times
		\sum_{0<|b|}\sum_{\substack{(m,4Na)=1\\m\text{ square -free}\\m>|b|/\Delta}}a_nn^{-k}\chi_{Nbq}({m})\overline{\epsilon_{m}}{m}^{\frac{-1}{2}}e\bigg(\frac{\overline{4N}bd}{m}\bigg)
		V\bigg(\frac{2\pi n}{a^2|d|q\sqrt{N}}\bigg).
	\end{align*}
	Just as in section 8 of \cite{iwaniec1990order},
	\begin{align*}
		  \sum_b \sum_m\ll a^{3}q^{2}r^{-\frac{3}{2}}j^{-3}
		(X+h)^{\epsilon-\frac{1}{2}}.\\
	\end{align*}
	With this estimate, it is straightforward to show
	\begin{align*}
		S_1
		&  \ll \sum_{\substack{a\leq A \\(a,4N)=1}}
		\sum_{\substack{r,j}}
		\sum_{q|j}
		\sum_{X/a^2q\leq| d | \leq (X+h)/a^2q}
		a^{3}q^{2}r^{-\frac{3}{2}}j^{-3}
		(X+h)^{\epsilon-\frac{1}{2}}\\
		&  \ll (X+h)^{\epsilon-\frac{1}{2}}h\sum_{\substack{a\leq A }}a
		\sum_{\substack{r,j}}r^{\frac{-3}{2}}j^{-3}
		\sum_{q|j}q
		\\
		&  \ll A^2(X+h)^{\epsilon-\frac{1}{2}}h.
	\end{align*}\end{proof}

	It remains to evaluate $S_0$. Because of the condition $b=0$, $S_0$ is written in the following simplified form 
	\begin{align*}
		S_0&=2\sum_{\substack{a\leq A \\(a,4N)=1}}\mu(a)
		\sum_{\substack{n=rj^2\\(n,a)=1}}a_nn^{-k}
		\sum_{q|j}\mu(q)\\&\times
		\sum_{\substack{dq\in \mathcal{D}\\ X/a^2q\leq |d|\leq (X+h)/a^2q}} 
		V(\frac{2\pi n}{a^2|d|q\sqrt{N}}).
	\end{align*} 
	We then evaluate the innermost sum. We split the inner sum into residue classes mod $4N$. Each class contributes 
	\begin{align*}
		\frac{1}{4N}\int_{X/a^2q}^{(X+h)/a^2q}V\bigg(\frac{2\pi n}{(a^2tq)\sqrt{N}}\bigg)dt+O\bigg(\bigg(1+\frac{n}{X}\bigg)^{-\epsilon}\bigg)
	\end{align*}
	by Euler's summation formula. Hence
	\begin{align*}
		S_0&=2\sum_{\substack{a\leq A \\(a,4N)=1}}\mu(a)
		\sum_{\substack{n=rj^2\\(n,a)=1}}a_nn^{-k}
		\sum_{q|j}\mu(q)\\&\times\bigg[\frac{\gamma(4N)}{4N}\frac{h}{a^2q}\int_{0}^{1}V\bigg(\frac{2\pi n}{(ht+X)\sqrt{N}}\bigg)dt+O\bigg(\bigg(1+\frac{n}{X}\bigg)^{-\epsilon}\bigg)
		\bigg].
	\end{align*} 
 	Here, $\gamma(4N)$ is the order of $\mathcal{D}$. The second term in the inner sum contributes $O(AX^{1/2+\epsilon})$, by trivial summation over $r,j$.

	As in section 9 of \cite{iwaniec1990order}, the first term in the innermost sum contributes
	\begin{align}\label{maintermcontribution}
		\gamma(4N)h\sum_{\substack{n=rj^2}}\frac{a_n\phi(j)}{2Nn^kj}
		\sum_{\substack{a\leq A \\(a,4Nj)=1}}\frac{\mu(a)}{a^2}
		\int_{0}^{1}V\bigg(\frac{2\pi n}{(ht+X)\sqrt{N}}\bigg)dt.
	\end{align}
	Now, using the identity 
	\begin{align*}
		&\sum_{\substack{a\leq A \\(a,4Nj)=1}}\frac{\mu(a)}{a^2}
		=\frac{6}{\pi^2}\prod_{p\mid 4Nj}(1-p^{-2})^{-1} +O(A^{-1}),\\
	\end{align*}
	we split (\ref{maintermcontribution}) into two. The first term is 
	\begin{align*}
		&\frac{3\gamma(4N)}{\pi^2N}\prod_{p|4N}(1-p^{-2})^{-1}h\sum_{\substack{n=rj^2}}\frac{a_n\phi(j)}{n^kj}\prod_{p|j}(1-p^{-2})^{-1}
		\int_{0}^{1}V\bigg(\frac{2\pi n}{(ht+X)\sqrt{N}}\bigg)dt\\
		&=\frac{3\gamma(4N)}{\pi^2N}\prod_{p|4N}(1-p^{-2})^{-1}h
		\int_{0}^{1}\bigg[\sum_{\substack{n=rj^2}}\frac{a_n}{n^k}\prod_{p|j}(1+p^{-1})^{-1}V\bigg(\frac{2\pi n}{(ht+X)\sqrt{N}}\bigg)\bigg]dt.\\
	\end{align*}
	Also, the second term contributes $O(A^{-1}hX^\epsilon)$, by trivial summation over $r,j$.

	Thus, finally,
	\begin{align*}
		S_0&=\frac{3\gamma(4N)}{\pi^2N}\prod_{p|4N}(1-p^{-2})^{-1}h
		\int_{0}^{1}\mathcal{B}(ht+X)dt\\
		&+O(A^{-1}hX^{\epsilon}+AX^{1/2+\epsilon}),
	\end{align*} 
	where 
	\begin{align*}
		\mathcal{B}(x)=\sum_{\substack{n=rj^2}}\frac{a_n}{n^k}\prod_{p|j}(1+p^{-1})^{-1}V\bigg(\frac{2\pi n}{x}\bigg).
	\end{align*}
	By shifting the line of integration to $(-\frac{1}{5})$, 
	\begin{align*}
		\mathcal{B}(x)
		&=\frac{1}{2\pi i}\int_{(4/5)}\frac{\Gamma(k+s)}{\Gamma(k)}\sum_{\substack{n=rj^2}}\frac{a_n}{n^{k+s}}\prod_{p|j}(1+p^{-1})^{-1}\bigg(\frac{2\pi }{x}\bigg)^{-s}\frac{ds}{s}\\
		&=-\frac{1}{2\pi i}\int_{(-1/5)}\frac{\Gamma(k+s)}{\Gamma(k)}\bigg(\frac{2\pi }{x}\bigg)^{-s}L_f(k+s)\frac{ds}{s}+L_f(k)\\
	\end{align*}
	
	Note that $L_f(s)$ appeared in \cite{iwaniec1990order} as $L(s)$ for $k=1$, and it was later generalized in \cite{luo1997determination} as $L_{f,1}(s)$ for all positive integer $k$. Recall that $L_f(k)\neq 0$ and $L_f(s)$  is polynomially bounded if $\Re(s)>k-1/4$. Thus we have
	\begin{align}
		\mathcal{B}(x)=L_f(k)+O_f(x^{-\frac{1}{5}}).\label{betabound}
	\end{align}
	From (\ref{betabound}), we obtain
	\begin{align*}
		S_0=\frac{3\gamma(4N)}{\pi^2N}\prod_{p|4N}(1-p^{-2})^{-1}L_f(k)h
		+O(h^{\frac{4}{5}}+A^{-1}hX^{\epsilon}+AX^{\frac{1}{2}+\epsilon}).
	\end{align*}
	Now we collect all the error terms of $S_f(X,h)$.  They are
	\begin{align*}
		O_{f,\epsilon}(h^{\frac{4}{5}}&+A^{-1}hX^\epsilon+AX^{\frac{1}{2}+\epsilon}+A^2(X+h)^{\epsilon-\frac{1}{2}}h\\&+ A^{-1-\epsilon}(X+h)^{\epsilon+\frac{1}{2}}h^{\frac{1}{2}}+ (X+h)^{\frac{1}{2}}hA^{-3+\epsilon}).
	\end{align*}
	
	Let $A=X^{\frac{1}{4}-\epsilon}$. Then for $ X^{3/4+\epsilon}\leq h\leq X$, the main term dominates the error terms. We have
	\begin{align*}
		S_f(X,h)&=\frac{3\gamma(4N)}{\pi^2N}\prod_{p|4N}(1-p^{-2})^{-1}L_{f}(k)h
		+O_{f,\epsilon}(hX^{-\epsilon})	.
	\end{align*}

	\section*{Acknowledgement}
	This work was supported by the National Research Foundation of Korea (NRF) grant funded by the Korea government(MSIT)(No. 2019R1A2C108860913, No. 2020R1A4A1016649). The author thanks his advisor Hae-Sang Sun for helpful comments and discussions. The author also thanks to Yoonbok Lee, Jaeseong Kwon and Seongjae Han for their  helpful comments and discussions.    
	
	\bibliography{reference}
	\bibliographystyle{plain}

\end{document}